\documentclass[a4paper,10pt]{amsart}

\usepackage[english]{babel}
\usepackage[T1]{fontenc}
\usepackage{mathrsfs}
\usepackage{amsthm}
\usepackage{amssymb}
\usepackage{amscd}

\theoremstyle{plain}
\newtheorem{thm}{\textbf{Theorem}}[section]

\newtheorem{prop}[thm]{Proposition}

\newtheorem{cor}[thm]{Corollary}

\newtheorem{lem}[thm]{Lemma}

\theoremstyle{definition}

\newtheorem{defi}[thm]{Definition}

\theoremstyle{remark}

\newtheorem{rem}[thm]{Remark}

\newtheorem{ex}[thm]{Example}

\newcommand{\ms}{\mathscr}
\newcommand{\mf}{\mathfrak}
\newcommand{\mc}{\mathcal}
\newcommand{\mb}{\mathbf}

\newcommand{\ov}{\overline}

\newcommand{\C}{\mc C}
\newcommand{\D}{\mc D}

\newcommand{\M}{\mc M}

\renewcommand{\t}{\mathsf{T}}			
\newcommand{\tu}{\t_\forall}			
\newcommand{\thu}{\t_{\forall\neg}}		

\newcommand{\crad}[1]{\sqrt[\C]{#1}}
\newcommand{\cprad}[1]{\sqrt[\C^+]{#1}}

\renewcommand{\sc}{S_\C}
\newcommand{\wt}{\mathcal W_\t}

\renewcommand{\AA}{\mathbb A}

\newcommand{\mo}{\models}
\newcommand{\onto}{\twoheadrightarrow}
\newcommand{\into}{\hookrightarrow}
\renewcommand{\iff}{\Leftrightarrow}
\newcommand{\imp}{\Rightarrow}

\renewcommand{\phi}{\varphi}

\renewcommand{\wt}{\widetilde}

\newcommand{\N}{\mathbf N}
\newcommand{\Z}{\mathbf Z}

\title{Geometrically closed rings}
\author{Jean Berthet}

\begin{document}

\begin{abstract}
We develop the basic theory of geometrically closed rings as a generalisation of algebraically closed fields, on the grounds of notions coming
from positive model theory and affine algebraic geometry. For this purpose we consider several connections between finitely presented rings and ultra\-pro\-ducts,
affine varieties and definable sets, and we introduce the key notion of an arithmetic theory as a purely algebraic version of coherent logic for rings.
\end{abstract}

\maketitle

\section*{Introduction}
 \subsection*{Arithmetics, geometry and logic}
  Algebraically closed (a.c) fields may be construed as ``arithmetically saturated'' domains, as they contain ``enough'' solutions for polynomial equations
  in one variable. The extension of this property to finitely many variables, through the geometric interpretation of Hilbert's Nullstellensatz, we
  alternatively conceive as a kind of ``geometric saturation''. Now arithmetics provide analogs to a.c closed fields, for example in the local
  fields $\mb R$ and $\mb Q_p$ and their generalisations, real closed and $p$-adically closed fields, or in situations arising from Galois theory, for instance infinite algebraic extensions of prime finite fields
  or certain extensions of number fields fixed by an absolute automorphism (see~\cite{MK}, Examples 1.4,5).
  In this kind of situation one often observes the conjunction of two related things : the existence of a ``relative'' Nullstellensatz from one hand,
  the share of a model theoretic property called ``model completeness'' on the other hand. A.c fields being the prototype of this connection, G.Cherlin
  exploited this example to derive a systematic relative Nullstellensatz in certain axiomatisable categories of ring extensions (\cite{GC}, Section III.6), which
  K.McKenna developped in \cite{MK} in a search for the determination of the ``$\t$-radical'' introduced by Cherlin, dealing with the fields case.
  As the limitation to ring extensions was imposed by classical model theory, it was our intuition that the interconnection between the two concepts was
  deeper and had to be extended to ring theory in general. Indeed, model completeness for fields has a stronger consequence called ``positive model
  completeness'', first introduced by A.MacIntyre for algebraico geometric applications of model theory in \cite{MI} and reinterpreted by
  I.Ben Yaacov and B.Poizat in ``positive model theory'' in \cite{BYP}, a (more ``algebraic'') generalisation of first order logic. Considering that
  ``coherent logic'', the categorical analogue of this last point of view, has a strong connection with algebraic geometry, we also had the feeling that an
  algebraic point of view on logic could be retained while avoiding the machinery of Grothedieck topologies.\\
  In the tradition of coherent logic (see \cite{BE1},\cite{ER} for example), the present work expresses the fruits of this meditation as the first stage of a
  program aiming at establishing algebraic connections between algebraic geometry and model theory. Our main objective here is to define a convincing
  generalisation of a.c fields in which one could develop an intrinsic ``relative affine algebraic geometry'', while a secondary goal on the course of it
  is the introduction for the algebraist of a purely algebraic form of some concepts and methods coming from logic, as they naturally appear in and prove
  to be connected to affine algebraic geometry or ring theory. In section~\ref{SEC1}, we characterise a.c fields among nontrivial rings by an algebraic
  translation of existential completeness and by Hilbert's Nullstellensatz (\ref{CAC}), and then explore the algebraic interplay between the two notions in
  a relativised context. In section~\ref{SEC2}, we translate existential completeness in the ``language'' of finitely presented rings and exploit this finiteness
  in the ultraproduct construction, where a ``finiteness property'' (\ref{FPP}) analogous to the compactness theorem of first order logic leads up to a
  ``global'' relative Nullstellensatz (\ref{POSNS}). In section~\ref{SEC3}, we introduce an algebraic version of formulas, axioms and theories of coherent
  or positive logic, which prove convenient to define certain relations and certain categories of rings coming from arithmetics, which we characterise in \ref{CARARITH}.
  In the last section~\ref{SEC4}, we introduce the ``arithmetic sets'' of a ring which may be used to characterise model complete theories (\ref{CARPMC}),
  the concept through which we eventually define geometrically closed rings, in which the ``definable sets'' of model theory have a structure property
  analogous to Chevalley's theorem on constructible sets. The rest of the introduction is devoted to the necessary background and notations. If $\C$ is a category, we will denote by $\C_0$ the collection of its objects and by $\C_1$ the collection of its
  morphisms.

  \subsection*{Affine algebraic geometry and reduced products}
  We will consider affine algebraic varieties in \emph{any ring}, even if there is no Zariski topology in general. If $X$ is any set, we will
  write $A[X]$ the polynomial ring with coefficients in $A$ and variables $X$ (and reserve small letters $x,y\ldots$, for individual variables).
  An \emph{affine algebraic set} in the affine space $A^X$ is the set $\ms Z(I)$ of zeros of an ideal of $A[X]$ and its  \emph{coordinate algebra}
  is the ring $A[X]/\ms I(\ms Z (I))$ understood with its $A$-algebra structure and embedded into the power
  $A^{\ms Z(I)}$ through the product of evaluation morphisms $(e_a)_{\ms Z(I)}:A[X]\to A^{\ms Z(I)}$. We will speak of an \emph{affine variety}
  if $X$ is finite and $I$ is finitely generated, denoting such an ideal as ``of finite type over $A$''. Affine varieties of $A$ and their regular morphisms
  are collected into the category $Aff(A)$, on which the coordinate algebra is a faithful contravariant functor. If $f:A\to B$ is an $A$-algebra, we will
  denote by $\ms Z_f(I)$ the zero set of $I$ in $B^X$ through $f$, and as usual by $\ms I(E)$ the set of polynomials in $A[X]$ vanishing on a subset $E$ of
  $B^X$.
  The geometric form of \emph{Hilbert's Nullstellensatz} states that if $k$ is an algebrically closed field, then for every ideal $I$ of
  finite type over $k$, we have $\ms I(\ms Z(I))=\sqrt I$, if $\sqrt I$ is the algebraic radical of $I$, i.e. the intersection of all prime
  ideals containing $I$.
  We will pay special attention to \emph{finitely presented (f.p) rings}, i.e. rings of the form $\Z[x_1,\ldots,x_n]/(p_1,\ldots p_m)$, and
  their (finitely cocomplete) small and full subcategory $\AA^{op}$, in which a canonical choice of tensor products will be implicit. 
  We will call a morphism $(m:P\to Q)\in \AA_1^{op}$ \emph{primitive}, and denote by $A_P$ the set $hom(P,A)$ of ring homomorphisms of $P$ into $A$,
  emphasing that we look at its elements as ``points'' of an affine space, writing $A_m : A_Q\to A_P$ the map induced by the (left exact) functor
  $hom(-,A)$ on $\AA$. We recall that the \emph{category $\ms D_A$ of elements of $A$} has elements $a\in \bigcup_{P\in \AA_0} A_P$ as objects and as
  arrows $u:( a\in A_P) \to (b\in A_Q)$ those $u:P\to Q$ such that $A_u(b)=bu=a$. The \emph{canonical diagram of $A$} is the forgetful functor
  $D_A:\ms D_A \to \AA^{op}$ and $(A,(a)_{\ms D_{A,0}})$ is a (canonical) colimit for $D_A$.
  If $(A_i)_{i\in I}$ is a family of rings and $\ms F$ is a filter of subsets of the index set $I$, $(\ms F,\supseteq)$ is a directed poset.
  As such it induces a directed diagram indexed by the products of $A_i$'s indexed themselves by \emph{elements of $\ms F$}, the morphism corresponding to an inclusion
  $S\subseteq T$ in $\ms F$ being the canonical projection $\pi^T_S : \prod_{i\in T} A_i \to \prod_{i\in S} A_i$. The (canonical) directed colimit
  is called the \emph{reduced product of the family $(A_i)_I$ modulo $\ms F$}, written $\prod_\ms F A_i$. The filter $\ms F$ is \emph{prime} if for any
  finite union of subsets of $I$ in $\ms F$, one of them is in $\ms F$, and this is equivalent to $\ms F$ being \emph{maximal}, in which case we speak of
  an \emph{ultrafilter} and an \emph{ultraproduct} for $\prod_\ms F A_i$. Notice that if one ring of the family is nontrivial, the ultraproducts of the
  family are nontrivial as well.\\
  Throughout this article, all rings will be commutative and unitary, $A$ will be one of them and $\C$ will denote a full subcategory of rings.

\section{Existentially closed noetherian rings}\label{SEC1}
  The classical model theoretic treatment of existential and model completeness deals with extensions of structures and was recently generalised
  to homomorphisms (see \cite{BYP} for instance).
  We will rely on the old vocabulary but the reader interested in model theory should keep in mind that everything here
  is ``positive'' - and algebraic. The first notion we introduce is an algebraic equivalent to the ``immersions'' of \cite{BYP},
  in order to characterise a.c fields among nontrivial rings.

  \begin{defi}
  A ring homomorphism $f:A\to B$ is \emph{(positively) existentially closed (e.c)} if for every ideal $I$ of finite type over $A$ such that
  $\ms Z_f(I)\neq\emptyset$, we have $\ms Z(I)\neq\emptyset$.\\
  $A$ is \emph{(positively) existentially closed in $\C$} if it is in $\C$ and every $A$-algebra in $\C$ is existentially closed; we note $\C^+$ the
  (full) subcategory of existentially closed rings of $\C$.
  \end{defi}

  \begin{thm}\label{CAC}
  If $\C$ is the category of nontrivial rings and $A\in\C_0$, then the following are equivalent.
  \begin{enumerate}
   \item $A$ is an algebraically closed field
   \item for every ideal $I$ in finitely many variables over $A$, $\ms I(\ms Z(I))=\sqrt I$
   \item for every ideal $I$ of finite type over $A$, $\ms I(\ms Z(I))=\sqrt I$
   \item $A$ is existentially closed in $\C$.
  \end{enumerate}
  \end{thm}
  \begin{proof}
  (1)$\imp$(2) is Hilbert's Nullstellensatz, while (3) is a particular case of (2).\\
  (3)$\imp$(1) If $a\in A$, we consider the ideal $I=(a)$, of finite type in the ``polynomial ring'' $A$ in no variables. If $a\neq 0$, we have
  $\ms Z(I)=\emptyset$, whence $\sqrt I\supseteq \ms I(\ms Z(I))=A$, and $1\in I$, so $a$ is invertible : $A$ is a field. Let $f:A\into K$ be
  an embedding of $A$ into an algebraically closed field. A non constant polynomial $p(x)$ in one variable over $A$ has a rational point in $K$,
  which means that $\sqrt p\neq A[x]$, so by hypothesis we have $\ms I(\ms Z(p))\neq A[x]$, i.e. $\ms Z(p)\neq\emptyset$ and $A$ is algebraically closed.\\
  (3)$\imp$(4) Suppose that (3) is satisfied, and that $I$ is an ideal of finite type over $A$ (say in the polynomial ring $A[X]$, where $X$ is a finite
  set of variables) with no rational point in $A^X$ : this means that $\ms Z(I)=\emptyset$, so we have $\sqrt I\supseteq\ms I(\ms Z(I))=A[X]$ by hypothesis,
  which means that $1\in \sqrt I$, hence $I=A[X]$ : the ideal $I$ cannot have a rational point in a non-trivial $A$-algebra, which means that $A$ is
  existentially closed in $\C$.\\
  (4)$\imp$(3) We distinguish two cases. First, as $A$ is not trivial, we have $\ms Z(A[X])=\emptyset$, so $\ms I(\ms Z(A[X]))=A[X]=\sqrt{A[X]}$
  : this deals with the case $I=A[X]$. Secondly, if $I\neq A[X]$ then $\sqrt I\neq A[X]$, so let $q\notin \sqrt I$ : there exists a prime ideal
  $\mf p\supseteq I$ and $q\notin \mf p$. In the fraction field $F$ of $A[X]/\mf p$, the classes of elements of
  $X$ collect into a rational point $c\in F^X$ for $I$ but not for $q$, whereby we may find an inverse $d\in A$ of $q(c)$. As $A$ is e.c as a non-trivial
  ring, we may find a rational point $a$ in $A$ for the ideal generated in $A[X,y]$ by $I$ and $q(X).y-1$ : the projection of $a$
  on $A^X$ is a rational point for $I$ while not for $q$, whence $q\notin \ms I(\ms Z(I))$ and we have $\ms I(\ms Z(I))\subseteq \sqrt I$.\\
  All four properties are then equivalent if we replace equalities in (2) and (3) by inclusions; but in this case $A$ is a field, so these are equalities.
  \end{proof}

  \begin{defi}
  If $I$ is an ideal of $A$, we will say that $I$ is
   \begin{enumerate}
    \item \emph{$\C$-prime} if $A/I$ embeds into a ring of $\C$
    \item \emph{$\C$-radical}, if $I=\crad I$, where $\crad I$ is the \emph{$\C$-radical of $I$}, i.e. the intersection of all
    $\C$-primes containing $I$.
   \end{enumerate}
  \end{defi}

  If $A$ is in $\C$ and $I$ is an ideal of $A[X]$, the rings $A[X]/\ms I(\ms Z(I))$ and $A[X]/\crad I$ lie in a category attached to $\C$ and strongly connected
  to the relativised Hilbert's property. We introduce them here as in \cite{HD} (Section 8.1) as a generalisation of \emph{quasivarieties}, well known to universal algebraists.
  \begin{defi}
  A full subcategory of rings $\mc S$ is \emph{special} if it is closed under products and embedded subrings.
  \end{defi}

  \begin{lem}
  The full subcategory $\sc$ of rings embeddable into a product of rings of $\C$ is the smallest special category containing $\C$. 
  \end{lem}
  \begin{proof}
  A smallest special category containing $\C$ clearly exists and contains $\sc$, so we only need to show that this last one is special. Call a
  ring \emph{$\C$-special} if it is in $\sc$. Any ring embedded into a $\C$-special ring is clearly $\C$-special, so let $(A_i)_I$ be a family
  of $\C$-special rings. For each $i$ we have by definition an embedding $f_i:A_i\into \prod_{J_i} B_j$, where $(B_j)_{j\in J_i}$ is a family
  of rings of $\C$. The $f_i$'s collect into a embedding $f:\prod_I A_i \into \prod_I \prod_{J_i} B_j$, whence $\prod_i A_i$ is in $\sc$, which
  is then closed under products and embedded subrings, i.e. is special.
  \end{proof}

  \begin{prop}\label{RAD}
  An ideal $I$ of $A$ is $\C$-radical if and only if $A/I$ is in $\sc$.
  \end{prop}
  \begin{proof}
  Consider the \emph{representation} of $A/I$ relatively to $\C$, which is by analogy with the case of prime ideals, the natural product homomorphism
  $f:A/I\to \prod_{\ms P_I} A/\mf p$, where $\ms P_I$ is the set of all $\C$-prime ideals containing $I$ : if $\pi_I$ denotes the canonical map
  $A\onto A/I$, we have $\crad I= Ker(f\pi_I)$. This means that if $I$ is $\C$-radical, then $A/I$ embeds into $\prod_{\ms P_I} A/\mf p$, which locates
  $A/I$ in $\sc$ by definition of $\C$-special rings.\\
  Conversely, if $A/I$ is in $\sc$, by the preceding lemma there exists an embedding $f:A/I\into \prod_J B_j$, where $B_j$ is a family of rings of $\C$. Denoting
  by $\mf p_j$ the kernel of the composite morphism $A\to A/I \to \prod_J B_j \to B_j$ obtained with the canonical projections, each $\mf p_j$ is a
  $\C$-prime containing $I$, whence $\crad I\subseteq \bigcap_J \mf p_j=I$ and $I$ is $\C$-radical.
  \end{proof}

  \begin{defi}
  We will say that $A$ is \emph{geometrically closed (g.c) in $\C$} if it is in $\C$ and satisfies the geometric form of Hilbert's Nullstellensatz
  relatively to $\C$, in other words if for every ideal $I$ of finite type over $A$, we have $\ms I(\ms Z(I))=\crad I$.
  \end{defi}
   In general, if $X$ is a finite set and $I\leq A[X]$ a finitely generated ideal, the coordinate algebra of the affine variety $V=\ms Z(I)$ has codomain
   $A[V]=A[X]/\ms I(\ms Z(I))$, which lies in $\sc$. If $A$ is noetherian, $A[V]$ is finitely presented as an $A$-algebra, so the coordinate
   algebra functor $A[-]:Aff(A)^{op}\to \sc$ has image in the category of finitely presented $\C$-special $A$-algebras. Now if in addition $A$ is geometrically
   closed in $\C$, by \ref{RAD} for every such algebra $f:A\to B$ there is a finite $X$ and a finitely generated $\C$-radical ideal $I$ such that
   $B\simeq A[X] /I=A[X]/\ms I(\ms Z(I))$, so $A[-]$ is a duality between affine varieties of $A$ and finitely presented $\C$-special $A$-algebras. We will
   now focus for a while on noetherian rings, after the following generalisation of Proposition 1.9 (e) of \cite{RH}.

  \begin{lem}\label{ZTOP}
  If $A$ is an integral domain and $E$ is a subset of $A^X$, then the Zariski closure of $E$ is the affine algebraic set $\ms Z(\ms I(E))$.
  \end{lem}
  \begin{proof}
  The set $\ms Z(\ms I(E))$ is closed and contains $E$. If now $F=\ms Z(I)\supseteq E$ is any closed subset of $A^X$ containing $E$, then $I\subseteq \ms I(E)$
  , whence $F=\ms Z(I)\supseteq \ms Z(\ms I(E))$. It follows that $\ms Z(\ms I(E))$ is the smallest closed set containing $E$.
  \end{proof}

  \begin{prop}\label{VARIRR}
  Suppose that $\C$ is a category of integral domains and $A$ is geometrically closed in $\C$ and artinian. If $V$ is an affine variety of $A$ in finite
  dimension, then $V$ is irreducible if and only if $\ms I(V)$ is $\C$-prime.
  \end{prop}
  \begin{proof}

  If $V=\ms Z(I)$ is irreducible, consider that $V$ is the union of the $\ms Z(\mf p)$'s, for all $\C$-prime ideals $\mf p$ containing $I$. As $A$
  is artinian, finitely many of them $\mf p_1,\ldots,\mf p_m$ suffice, so we have $V=\bigcup_{k=1}^m \ms Z(\mf p_k)$. Now as $V$ is irreducible, in fact
  one of them we call $\mf p$ suffices and $V=\ms Z(\mf p)$, whence $\ms I(V)=\ms I(\ms Z(\mf p))=\crad{\mf p}$ (by geometric closedness) $=\mf p$
  (as $\mf p$ is $\C$-radical), so $\ms I(V)$ is $\C$-prime.\\
  The converse needn't the hypothesis of geometric closedness. If $I=\ms I(V)$ is $\C$-prime and $V=F_1\cup F_2$ (where $F_i=\ms Z(J_i), i=1,2$), we
  have $V=\ms Z(J_1J_2)$ (so $I=\ms I(\ms Z(J_1J_2))$), because $A$ is a domain. Now by definition of $\C$-primality, $A[X]/I$ embeds into a
  ring $B$ of $\C$, so $I$ is prime. As $J_1J_2\subseteq I$, one of the $J_i$'s, say $J_1$, is then included in $I$ and we have 
  $F_1\supseteq \ms Z(\ms I(V))=\overline V$ (by the lemma) $=V$ : $V$ is irreducible.
  \end{proof}

  The category of rings embeddable into a ring of $\C$ we will note $U_\C$ and call the \emph{universal category} associated to $\C$. If $\C$ is a
  category of fields, the proposition entails that the function fields of irreducible varieties of the fields in $\C$ are also in $U_\C$. In general
  the categories $\C$ and $U_\C$ obviously generate the same special category, so geometric closedness relative to $U_\C$ is the same as relative to $\C$ or
  even to $\sc$, provided that the rings are in the suitable categories. Rings in $U_\C$ and nontrivial rings in $\sc$ have in common to be in the category $\C^-$ of rings $B$
  which \emph{continue} into a ring of $\C$, i.e. for which there exists an homomorphism $f:B\to C$, with $C$ in $\C$. In the light of Theorem~\ref{CAC},
  if $\C$ is the category of a.c fields, $\C^-$ is the category of nontrivial rings. We will replace $\C$ by any category of fields after the following
  example.
  
  \begin{ex}
  Remember that a field $k$ is real if and only if $-1$ is not a sum of squares in $k$, which is equivalent to there existing a total
  order on $k$, compatible with the ring structure. An ideal $I$ of $A$ is real prime if $A/I$ embeds into a real field and the real radical $\sqrt[R]{I}$ of $I$ is the
  intersection of real primes containing $I$ (see \cite{BCR}, chapters 1 and 4).
  These are the relativised notions we introduced for the category $\C$ of real fields and $\C^-$ is the category of what we call a \emph{pre-real}
  ring, i.e where $-1$ is not a sum of squares, or equivalently where there exists a real prime ideal. We also recall that a real field $k$ is \emph{real closed}
  if it has no proprer real algebraic extension, which is characterised by two properties : every element of $k$ or its opposite is a square in $k$,
  and every polynomial equation in one variable with coefficients in $k$ and odd degree has a root in $k$. In this situation,
  $U_\C$ is the category of what we name \emph{real integral domains}, i.e domains in which every zero sum of squares is trivial.
  \end{ex}

  \begin{thm}\label{ECNOET}
  If $\C$ is a category of fields and $A$ is an existentially closed ring of $\C^-$, then $A$ is a geometrically closed field in $U_\C$.
  \end{thm}
  \begin{proof}
  Let $I$ be an ideal of finite type over $A$ (say of $A[X]$ for $X$ finite), where $A$ is existentially closed in $\C^-$ : we want to show that
  $\ms I(\ms Z(I))\subseteq \crad I$ so let $q\notin \crad I$. There exists a $\C$-prime ideal $\mf p$ of $A[X]$ containing $I$ but not $q$,
  and we consider the composite morphism $f:A\into A[X] \onto A[X]/\mf p \into B$, with $B$ in $\C$ by hypothesis on $\mf p$.
  As $A$ is existentially closed in $\C^-$, $f$ itself is e.c and as $B$ is a field,
  as in the proof of Theorem~\ref{CAC} $A$ itself is a field. The ideal $\mf p$ is then finitely generated and there exists a rational point for $\mf p$ in $A$
  as in~\ref{CAC}, which is a zero for $I$ but not for $q$. We then have $q\notin\ms I(\ms Z(I))$, and the reverse inclusion comes from the fact that $A$ is in $U_\C$.
  \end{proof}
 
  This tells us where to find g.c noetherian rings in certain cases, but in contrast with geometric closedness, if $A$ is existentially closed in $\C$
  and $\C$ contains a trivial ring $\mb 1$, then the unique homomorphism $A\to \mb 1$ is e.c, hence  $A$ is trivial. This means that working with
  ``interesting'' e.c rings somehow involves avoiding trivial rings. In this case, we get a stronger connection between the two notions, as
  Theorem~\ref{CAC} illustrates in the ``absolute'' case. This involves a ``local'' definition of geometric closedness. We note $\C^*$ the full subcategory
  of nontrivial rings of $\C$. 
  \begin{defi}
  An homomorphism $f:A\to B$ is \emph{geometrically closed} if for every ideal $I$ of finite type over $A$, $\ms I(\ms Z(I))\subseteq\ms I(\ms Z_f(I))$.
  \end{defi}
 
  \begin{lem}
  $A$ is geometrically closed in $\C$ if and only if every $A$-algebra in $\C$ is geometrically closed.
  \end{lem}
  \begin{proof}
  If $I$ is an ideal of finite type over $A$, say $I\leq A[X]$, a polynomial $p\in A[X]$ is in $\crad I$ if and only it is in $\ms I(\ms Z_f(I))$ for
  every homomorphism $f:A\to B$ in $\C$.
  \end{proof}

  \begin{lem}\label{IMMGEO}
  A geometrically closed homomorphism into a nontrivial ring is existentially closed.
  \end{lem}
  \begin{proof}
  Let $f:A\to B$ be such a homomorphism and $I\subseteq A[X]$ an ideal of finite type over $A$, such that $\ms Z_f(I)\neq\emptyset$. Let $b\in \ms Z_f(I)$ and
  $e_b:A[X]\to B$ be the evaluation morphism at $b$ : by hypothesis on $B$ the ideal $Ker(e_b)$ cannot be the whole ring $A[X]$, whence $\ms I(\ms Z_f(I))\neq A[X]$.
   By hypothesis on $f$ we then have $\ms I(\ms Z(I))\neq A[X]$, which means that $\ms Z(I)\neq\emptyset$, and $f$ is e.c.
  \end{proof}

  \begin{cor}\label{CORNOETH}
  If $\C$ is a category of fields, then $A$ is existentially closed in $\C^-$ if and only if it is geometrically closed
  in $U_{\C}$ (and then is a field).
  \end{cor}
  \begin{proof}
  By \ref{ECNOET}, we only need to proof that if $A$ is g.c in $U_\C$, then $A$ is e.c in $\C^-$. If $f:A\to B$ is a homomorphism into a ring
  of $\C^-$, then we may continue by $g:B\to C$, with $C$ in $\C$. By hypothesis, the homomorphism $gf$ is geometrically closed and as a field is not
  trivial, by the preceding lemma it is e.c, whence $f$ is e.c as well, and $A$ is existentially closed in $\C^-$.
  \end{proof}
  
  \begin{cor}\label{RCF}
  The existentially closed pre-real rings are the real closed fields.
  \end{cor}
  \begin{proof}
  If $A$ is existentially closed as a pre-real ring, there is an homomorphism $f:A\to k$ into a real closed field $k$, because every real field embeds
  into a real closed field. The ring $A$ itself is obviously a field and if $a\in A$, one of $f(a)$ or $-f(a)$ is a square in $k$, a property which is
  clearly reflected through $f$ by existential completeness. A similar argument shows that every polynomial $y^n + a_0y^{n-1} +\ldots +a_n$ of odd degree $n$
  with coefficients in $A$ has a root in the real closed field $k$, hence in $A$ by e.c, so $A$ itself is real closed.\\
  Reciprocally, let $\C$ be the category of real fields and $k$ a real closed field. As the real radical is our $\C$-radical, by the real Nullstellensatz
  (\cite{BCR}, Theorem 4.1.4) $k$ is geometrically closed in $U_\C$, hence by the preceding corollary it is existentially closed in $\C^-$, the category of
  pre-real rings.
  Notice that all this also shows that the real closed fields are the existentially closed real fields.
  \end{proof}

\section{Ultraproducts and finiteness}\label{SEC2}
  Corollary~\ref{CORNOETH} connects the two notions of existential completeness and geometric closedness in the manner of our
  characterisation \ref{CAC} of a.c fields. The finiteness hypothesis in the definition of e.c homomorphisms was exploited there through noetherianity,
  but may be invested in another way to give us a deeper understanding of the connection. In the preceding analysis $A$ relates to the ambient category
  $\C$, and we could define $A$ being (intrinsically) ``geometrically closed'' if it would be so in the special category $S_A$ of rings embedded into a
  power of $A$, as it contains information about the affine geometry of $A$. This in particular would mean that $A$ is e.c among non-trivial members of
  $S_A$, but this is not clear where it would lead us in this generality, and if it would provide a mean to discriminate the irreducible varieties, as the
  category $U_{S_A}$ is $S_A$ itself. We rather engage in the study of certain categories of rings in which we will replace the ``local'' hypothesis of
  noetherianity by a ``global'' finiteness hypothesis related to ultrafilters.
  
  \begin{defi}
  If $m:P\to Q$ is primitive, we note $\exists_m A$ the set $A_m(A_Q)$ of elements $a\in A_P$ such that there exists $c\in A_Q$ with $cm=a$.
  \begin{enumerate}
    \item A primitive morphism \emph{with parameter in $A$} is a pair $(m,a)$, where $m:P\to Q$ is primitive and $a\in A_P$ is the parameter
    \item If $X$ is a set of primitive morphisms with parameters among the objects of a subcategory $I$ of $\ms D_A$, a cocone $(B,(f_b)_{b\in I_0})$ for
    $D_{A\lceil I}$ \emph{realises (resp. avoids) $X$} if for every $(m,a)\in X$  we have $f_a\in \exists_m B$ (resp. $f_a\notin\exists_m B$)
    \item In case $f:A\to B$ is a ring homomorphism, we will say as well that \emph{$f$ realises (resp. avoids) $X$}, if the cocone
    $(B,(f\circ a)_{a\in \ms D_A})$ realises (resp. avoids) $X$.
  \end{enumerate}
  \end{defi}
 
  In~\cite{AR} we have the definition (2.27) of a \emph{pure} (ring) homomorphism $f:A\to B$, i.e such that for every $a\in A_P$ and $b\in A_Q$ with $fa=bm$,
  there exists $c:Q\to A$ such that $cm=a$. In other words, $f$ is pure if and only if $f^{-1} (\exists_m B)=\exists_m A$ for every primitive $m$, and
  they characterise as follows (see~\cite{AR}, 5.15 for a logical version compatible with~\cite{BYP}).
  
  \begin{lem}\label{CARIM}
  A ring homomorphism $f:A\to B$ is pure if and only if it is existentially closed.
  \end{lem}
  \begin{proof}
  An $A$-algebra $g:A\to C$ is finitely presented as such if and only if it is isomorphic to a tensor product of the form $a\otimes m:A\to A\otimes_P Q$,
   where $(m:P\to Q,a)$ is a primitive morphism with parameter in $A$. Now $f$ is existentially closed if and only if for every diagram of the form
   \begin{displaymath}
    \begin{CD}
     A @>g>> C\\
     @| @VVhV\\
     A @>>f> B
    \end{CD}
   \end{displaymath}
  where $g$ is a finitely presented $A$-algebra, $g$ has a retraction, which is equivalent to saying, by an easy diagram chasing, that $f$ is pure.
  \end{proof}

  The ``finiteness'' of primitive morphisms may be combined with ultrafilters to get the following analogue of the compactness theorem of first order logic,
  which may be found in \cite{HD} as Theorem 6.1.1, or in a ``semantic version'' closer to what is expounded here as Theorem 4.5 of \cite{BP}. This
  ``finiteness property'' will enable us to derive a relative Nullstellensatz in
  the category $\C^+$ of e.c rings of a given category $\C$, when this last is closed under ultraproducts, and to get a ``global'' version
  (\ref{CARGEC}) of Corollary~\ref{CORNOETH}. If $f:c\to d$ is a morphism in a category $\D$, we write for convenience $\langle f \rangle$ to denote
   the subcategory generated by $f$.
  
  \begin{thm}\label{FPP}
  Let $X$ and $Y$ be two sets of primitive morphisms with parameters in $A$, such that for all finite subsets $X_0$ of $X$ and $Y_0$ of $Y$ and every
  finite subcategory $I$ of $\ms D_A$ containing their parameters, there exists a cocone for $D_{A\lceil I}$ with vertex in $\C$, realising $X_0$ and
  avoiding $Y_0$. Then there exists an homomorphism $f:A\to B$ into an ultraproduct of rings of $\C$, realising $X$ and avoiding $Y$.

  \end{thm}
  \begin{proof}
  We first build the ultraproduct $B$ as the vertex of a cocone for $D_A$. Consider the set $\Theta$ of triples $(I,x,y)$, where $x$ is a finite subset of $X$, $y$ is a
  finite subset of $Y$ and $I$ is a finite subcategory of $\ms D_A$ such that $I_0$ contains the parameters of the elements of $x\cup y$. For $(I,x,y)
  \in\Theta$, set $\Theta_{I,x,y}=\{(J,u,z)\in\Theta: I\subseteq J, x\subseteq u, y\subseteq z\}$ : their collection is a filter basis of subsets of
  $\Theta$, contained by the axiom of choice in an ultrafilter $\ms U$. By hypothesis, choose for every triple $(I,x,y)\in \Theta$ a ring $B_{I,x,y}$ in
  $\C$, which is the vertex of a cocone $(B_{I,x,y},(f_{a,I,x,y})_{a\in I_0})$, realising $x$ but avoiding $y$ : we consider the ultraproduct
  $B=\prod_\ms U B_{I,x,y}$. Now for every element $a\in A_P$ of $\ms D_A$, consider also the morphism $f_a: P=D_A(a)\to B$, defined by $f_a=\pi_a \circ
  (f_{a,I,x,y})_{(I,x,y)\in \Theta_a}$, where $\Theta_a=\{(I,x,y) : a\in I_0\}=\Theta_{\langle1_a\rangle,\emptyset,\emptyset}$ and $\pi_a : \prod_{\Theta_a} B_{I,x,y}\to B$ is the transition morphism
  of the ultraproduct. If $u:a\in A_P\to b\in A_Q$ is a morphism in $\ms D_A$, then considering
  $\Theta_u=\{(I,x,y):u\in I_1\}=\Theta_{\langle u\rangle,\emptyset,\emptyset}$, one readily sees that $(B,(f_a)_{\ms D_{A,0}})$
  is a cocone for $D_A$, whence there exists a unique homomorphism $f:A\to B$ such that $f_a=fa$ for every element $a:P\to A$.\\
  Secondly, we show that $f$ realises $X$ : let $(m:P\to Q,a)\in X$ and consider $\Theta_{(m,a)}=\{(I,x,y) : (m,a)\in x\}=\Theta_{\langle 1_a \rangle,
  \{(m,a)\},\emptyset}$. By choice of the $B_{I,x,y}$'s, for
  every $(I,x,y)\in \Theta_{(m,a)}$, we may choose a morphism $f_{(m,a),I,x,y}:Q\to B_{I,x,y}$ such that $f_{(m,a),I,x,y}\circ m = f_{a,I,x,y}$ (whereas
  there is no such for the elements of $Y$). Let us note $\prod f_{(m,a)}$ the family $(f_{(m,a),I,x,y})_{(I,x,y)\in\Theta_{(m,a)}}$ and
  set $f_{m,a}=\pi_{(m,a)}\circ \prod f_{(m,a)}$, where $\pi_{(m,a)}$ is the transition morphism
  $\prod_{\Theta_{(m,a)}} B_{I,x,y}\to B$. We have $f_{(m,a)}\circ m=\pi_{(m,a)}\circ \prod f_{(m,a)}\circ m= \pi_{(m,a)}
  \circ (f_{a,I,x,y})_{\Theta_{(m,a)}}=f_a=fa$ : we have $f_a\in \exists_m B$, which means that $f$ realises $(m,a)$, and $f$ realises $X$.\\
  Finally, we show that $f$ avoids $Y$. Suppose that $(m,a)\in Y$ but we have a commutative square of the form :
  \begin{displaymath}
   \begin{CD}
    P @>m>> Q\\
    @VaVV @VVbV \\
    A @>>f> B.
   \end{CD}
  \end{displaymath}
  As $Q$ is finitely presented and $B$ is a directed colimit, there exists $S\in \ms U$ and a family $b_S=(b_\theta)_S$ of morphisms $b_\theta:Q\to
  B_\theta$ such that $b=\pi_S\circ b_S$, where $\pi_S$ is the transition morphism of the ultraproduct.
  Introducing $\Theta^*_{(m,a)}=\Theta_{\langle 1_a \rangle,\emptyset,\{(m,a)\}}$, let $T=S\cap\Theta^*_{(m,a)} \in \ms U$. If $\pi_T$ is the corresponding transition morphism, we have
  $\pi_T\circ (b_\theta \circ m)_T = \pi_T\circ \pi^S_T \circ b_S \circ m= bm=fa =\pi_T\circ (f_{a,\theta})_{\theta\in T}$, where $f_{a,\theta}=f_{a,I,x,y}$ if $\theta=(I,x,y)$,
  the last equality being true by definition of $f$, as $T\subset\Theta^*_{(m,a)}\subset \Theta_a$. By definition of $B$ as a directed colimit, as
  $P$ is finitely presented this means that there exists a $T'\in \ms U$ such that $T'\subset T$ and 
  $(b_\theta\circ m)_{\theta\in T'}=(f_{a,\theta})_{\theta\in T'}$. As $T'$ is not empty
  we pick up $\theta\in T'$, for which we have $b_\theta\circ m=f_{a,\theta}$, hence $f_{a,\theta}\in \exists_m B_\theta$, and this
  contradicts the fact that $f_{a,\theta}\notin \exists_m B_\theta$, as $T'\subset \Theta_{(m,a)}^*$. We conclude that $f$ does in fact avoid $Y$.  
  \end{proof}

  We now introduce into this algebraic setting some analogues of notions taken from positive model theory. The following lemma underlies all that follows
  and is reminiscent of Lemma 14 in \cite{BYP}. It will be convenient to adapt here the notion of resultant found in
  \cite{BEL} (Definition 9) and inspired from \cite{HD} (Section 8.5).
  \begin{defi}
  If $m:P\to Q$ is a primitive morphism, the \emph{resultant of $m$ modulo $\C$}, in short $Res_\C [m]$, is the set of all primitive morphisms
  $n:P\to R$ such that in every ring $B$ of $\C$, we have $\exists_m B \cap \exists_n B=\emptyset$.
  \end{defi}

  \begin{lem}\label{RESCOMP}
  If $\C$ is closed under ultraproducts, $A$ is existentially closed in $\C$ if and only if for every primitive $m:P\to Q$ we have
  $A_P - \exists_m A =\bigcup \{\exists_n A : n\in Res_\C [m]\}$. 
  \end{lem}
  \begin{proof}
  We first show that the condition is necessary. Let $a\notin\exists_nA$ for every $n\in Res_\C [m]$ : we want to show that $a\in\exists_mA$. In order to achieve this, we are
  going to realise the couple $(m,a)$ in an $A$-algebra with codomain in $\C$ and then come back in $A$ by existential completeness.
  Suppose by way of contradiction that this is not possible; by the Finiteness Property~\ref{FPP}, there exists a finite subcategory $I$ of $\ms D_A$ such
  that $a\in I_0$ and $(m,a)$ has no common realisation with $I$ in $\C$. Let $B\in \C_0$ and $g_a\in B_P$; if $(B_I,(f_b)_{b\in I_0})$ is a
  colimit for $D_{A\lceil I}$ in $\AA^{op}$ and $g:B_I\to B$ is such that $gf_a=g_a$, we may set $g_b=gf_b$ for all $b\in I_0$ to get a cocone
  $(B,(g_b)_{b\in I_0})$ for $D_{A\lceil I}$. If we had $g_a\in \exists_m B$, there would be a $c\in B_Q$ such that $cm=g_a=gf_a$, as in the following
  diagram :
   \begin{displaymath}
    \begin{CD}
   P @>m>> Q\\
   @Vf_aVV @VVcV\\
   B_I @>>g> B.
    \end{CD}
   \end{displaymath}
  This would yield a common realisation of $(m,a)$ and $I$ in $(B,(g_b)_{I_0})$, which is excluded. We conclude that $g_a\notin\exists_m B$, whence
  $f_a\in Res_\C[m]$, because $B_I$ is finitely presented. Now, as $(A,(b)_{b\in I_0})$ is a cocone for $D_{A\lceil I}$ we have the following commutative diagram of rings :
   \begin{displaymath}
    \begin{CD}
     P @>f_a>> B_I\\
     @VaVV @VV!V,\\
     A @= A,
    \end{CD}
   \end{displaymath}
  a contradiction to the hypothesis that $a\notin \exists_{f_a} A$. All this shows that there exists a continuation $f:A\to B$ in $\C$ realising $(m,a)$;
  in other words, there exists a $b\in B_Q$ such that $bm=fa$. As $A$ is e.c in $\C$ we have $a\in \exists_m A$ by Lemma~\ref{CARIM} and the ``only if''
  direction if proved.\\
  Reciprocally, suppose that the complement of every set defined in $A$ by a primitive morphism $m$ is defined by $Res_\C[m]$ and let
  $f:A\to B$ be an homomorphism in $\C$. If $m:P\to Q$ is primitive and $a\in A_P$ is such that $fa\in \exists_m B$, then for every $n\in Res_\C[m]$ we
  have $b\notin \exists_n B$ by definition of the resultant. As $f$ is an homomorphism, this means we have $a\notin \exists_n A$, whence by hypothesis
   we get $a\in \exists_m A$, and $f$ is an immersion : $A$ is existentially closed in $\C$, and the proof is complete.
  \end{proof}
  
  \begin{thm}\label{POSNS}
  If $\C$ is closed under ultraproducts, every ring in $\C^+$ is geometrically closed in $\C^+$.
  \end{thm}
  \begin{proof}
  Let $X$ be a finite set and $I$ a finitely generated ideal of $A[X]$ : we want to show that $\ms I(\ms Z(I))\subseteq\cprad I$, so let $q\notin\cprad I$.
  By definition, there exists a $\C^+$-prime $\mf p$ containing $I$ and not $q$, so we have a ring $B$ of $\C^+$ and a composite morphism
  $f:A\to A[X]/\mf p \into B$, where the last is an embedding and the class $b$ of $X$ in $B$ is a solution of $I$ which avoids $q$, so we have
  $\ms Z_f(I)-\ms Z_f(q)\neq\emptyset$. As $I$ is of finite type, there exist two primitive morphisms with parameters in $A$, say $(m:P\to Q,a)$ and
  $(n:R\to S,b)$ and two pushouts as pictured in the following diagrams :
   \begin{displaymath}
    \begin{CD}
     P     @>m>> Q     \\
     @VaVV       @VVvV \\
     A     @>>u>  {A[X]}/I
    \end{CD}\hspace{1cm}    
    \begin{CD}
     R     @>n>> S \\
     @VbVV       @VVyV\\
     A     @>>x>  {A[X]}/{(q)}. 
    \end{CD}
   \end{displaymath}
  The condition ``$\ms Z_f(I)-\ms Z_f(q)\neq\emptyset$'' is then equivalent to ``$fa\in \exists_m B$ and $f\circ (a\otimes b)\notin \exists_{m\otimes n} B$'',
  where $m\otimes n :P\otimes R \to Q\otimes S$ and $a\otimes b : P\otimes R\to A$ are the homomorphisms induced by the tensor products. As this condition
  is satisfied and $B$ is e.c in $\C$, by the preceding lemma there is an $(l:P\otimes R\to S)\in Res_\C [m\otimes n]$ such that $f\circ (a\otimes b)
  \in \exists_l B$. Now as $A$ itself is e.c in $\C$, $f$ is existentially closed and by Lemma~\ref{CARIM} we have $a\in \exists_m A$ and
  $a\otimes b\in \exists_l A$, whence by Lemma~\ref{RESCOMP} again we have $a\otimes b\notin \exists_{m\otimes n} A$, which means that
  $\ms Z(I)-\ms Z(q)\neq\emptyset$ and $q\notin \ms I(\ms Z(I))$ : $A$ is geometrically closed in $\C^+$.
  \end{proof}

  \begin{defi}\label{DEFPMC}
  We will say that $\C$ is \emph{pseudo model complete (p.m.c)} if every homomorphism in $\C$ is existentially closed, i.e if $\C=\C^+$.
  \end{defi}

  \begin{cor}\label{CARGEC}
  If $\C$ is closed under ultraproducts, then $\C^*$ is pseudo model complete if and only if every ring in $\C^*$ is geometrically closed.
  \end{cor}
  \begin{proof}
  If $\C^*$ is p.m.c, then we have $\C^*=(\C^{*})^+$ and this last category is also closed under ultraproducts, so by the theorem every ring
  in $\C^*$ is geometrically closed.\\
  Reciprocally, if $f:A\to B$ is an homomorphism in $\C^*$, by hypothesis $f$ is geometrically closed so by Lemma~\ref{IMMGEO} it is an
  immersion, i.e. $\C^*$ is p.m.c.
  \end{proof}

\section{Arithmetic theories}\label{SEC3}
  The graph of the equality relation in $A$ is ``defined'' by $e:\Z[x,y]\to \Z[x,y]/(x-y)$, in the sense that $a=b$ if and only if $(a,b)\in
  \exists_e A$, and if $A$ is a field, the graph of the \emph{inequality} relation is defined by the primitive morphism
  $i:\Z[x,y]\to \Z[x,y,z]/(z(y-x)-1)$. Consider now the (real closed) field $\mb R$ : its nonnegative elements are precisely
  the squares, so we have $a\leq b$ if and only if there exists $c$ such that $b-a=c^2$ and the ordering $\leq$ has graph $\exists_o \mb R$,
  for the primitive
  $o:\Z[x,y]\to\Z[x,y,z]/(z^2-y+x)$. Alternatively, we could have defined its complement, the strict order $>$ which graph is given by
  $s:\Z[x,y]\to \Z[x,y,z]/(z^2(x-y)-1)$. Those and several relations naturally arising in rings and fields may be ``defined'' by primitive morphisms.
  This is the occasion to introduce the algebraic counterpart of first order logic's formulas, and the key notion of an ``arithmetic theory'', an
  analogue in $\AA^{op}$ of ``coherent'' (see \cite{MR}) or ``h-inductive'' (see \cite{BYP}) theories of rings.

  \begin{defi}
  Let $P$ be a finitely presented ring.
   \begin{enumerate}
    \item An \emph{arithmetic formula (on $P$)} will be a finite set $F$ of primitive morphisms with same domain ($P$)
    \item An \emph{arithmetic axiom} will be a couple $\chi=[P,F]$, where $F$ is a formula on $P$, and $A$ will be a
    \emph{model of $\chi$} (noted $A\mo\chi$) if $A_P=\bigcup_{m\in F} \exists_m A$
    \item An \emph{arithmetic theory} will be a set $\t$ of arithmetic axioms and $A$ will be a \emph{model} of $\t$ if $A\mo\chi$ for every $\chi\in\t$.
   \end{enumerate}
  An \emph{axiomatisation} of a category $\C$ is an arithmetic theory $\t$ such that $\C=\M(\t)$ is the full subcategory of models of $\t$, and such a
  category we call itself \emph{arithmetic}. 
  \end{defi}
  
  If $\ov x=x_1,\ldots,x_m$ and $\ov y=y_1,\ldots,y_n$ are finite tuples of variables and $\ov f=f_1,\ldots,f_p$ and $\ov g=g_1,\ldots,g_q$ are finite
  tuples of polynomials, we say that the primitive morphism $\Z[\ov x]/(\ov f)\to \Z[\ov x,\ov y]/(\ov f,\ov g)$ is \emph{normal}. An
  arithmetic formula will be \emph{normal} if its elements are, the same for an axiom. If $P=\Z[\ov x]/(\ov f)$ and $F=\{m_j:\Z[\ov x]/(\ov f)
  \to \Z[\ov x,\ov y_j]/(\ov f,\ov g_j), j=1,\ldots,k\}$, where $\ov g_j=g_{j1},\ldots,g_{jl_j}$, we will write
  \begin{displaymath}
  [P,F]=[\forall \ov x\ (\bigwedge_{i=1}^p f_i(\ov x)=0) \imp \bigvee_{j=1}^k \exists\ov y_j (\bigwedge_{l=1}^{l_j} g_{jl}(\ov x,\ov y_j)=0)].
  \end{displaymath}
  as an abbreviation, as this will make the presentation of examples and the translation of intuitive axiomes easier (this is the reverse approach to
  the translation of coherent axioms into morphisms, as presented in \cite{BE1}). If $F=\emptyset$, we will
  occasionally write $[\forall \ov x\ (\bigwedge_{i=1}^p f_i(\ov x)=0) \imp \bot]$ and if $p=0$, we will write $\top$ instead of the empty conjunction
  on the left of the implication.
  \begin{rem}
  We emphasise the fact that those ``arithmetic'' formulas and theories avoid the use of any first order logic proper and that the notation with
  quantifiers and boolean operators is only a matter of typing and translating intuition into arithmetic formulas.
  \end{rem}
  
  \begin{ex}
  \begin{enumerate}
   \item Non-trivial rings are the models of the axiom $\chi=[\mb 1,\emptyset]$. Rings of characteristic $n\geq 2$ are the models of the axiom
   $\chi_n=[\Z,\{\Z\to \Z/n\Z\}]=[\top\imp (n=0)]$. Rings of characteristic $0$ are the models of the infinite theory
   $\chi_0=\{[\Z/n\Z,\emptyset]=[(n=0)\imp\bot] : n\in \N^*\}$.
   \item Integral domains are the models of the theory $\t_{id}$ which axioms are $\chi$ and
   $[\forall x,y \ (xy=0) \imp (x=0)\vee (y=0)]=[\Z[x,y]/(xy),\{\Z[x,y]/(xy)\to \Z[x,y]/(x), \Z[x,y]/(xy)\to \Z[x,y]/(y)\}]$. Reduced rings are the models
   of the theory $\t_{rr}$ which axioms are all $[\forall x \ (x^n=0)\imp (x=0)]$, for $n\in \N^*$.
   \item Fields are the models of $\t_f=\t_{id}\cup \{[\forall x\ \top\imp (x=0)\vee \exists y(xy-1=0)]\}$ and algebraically closed fields are the models
   of $\t_{acf}=\t_f\cup \{[\forall x_1,\ldots,x_n\ \top\imp\exists y(y^n+x_1y^{n-1} +\ldots + x_n=0)] : n\in \N^*\}$.
  \end{enumerate}
  \end{ex}

  \begin{prop}\label{DIAMOR}
  Let $X=\{(m_k:P_k\to Q_k,a_k), k=1,\ldots,n \}$ be a finite set of primitive morphisms with parameters in $A$ and $I$ a finite subcategory of $\ms D_A$
  containing the parameters of $X$. If $(P^*,(i_a)_{a\in I_0})$ is a colimit for $I$ in $\AA^{op}$ and $m_k^*$ is obtained from $m_k$ by the change of basis
  $i_{a_k}:P_k\to P^*$ in $\AA^{op}$ , then for every ring $B$ we have $B\mo [P^*,\{m_k^*\}] \iff$ every cocone for $I$ with vertex $B$ realises one of
  the $(m_k,a_k)$'s.
  \end{prop}
  \begin{proof}
  It will be convenient to draw a picture of the change of basis for every $m_k$ as the following pushout :
  \begin{displaymath}
   \begin{CD}
    P_k @>m_k>> Q_k\\
    @Vi_{a_k}VV @VVj_{a_k}V\\
    P^* @>>m_k^*> Q_k^*.
   \end{CD}
  \end{displaymath}
  Suppose that $B\mo [P^*,\{m_k^*\}]$ and $(B,(g_a)_{a\in I_0})$ is a cocone for $I$. By definition of $P^*$, there is a unique $g:P^*\to B$ such that
  $gi_a=g_a$ for every $a\in I_0$. By hypothesis, as $g\in B_{P^*}$ there exists $k\in\{1,\ldots,n\}$ such that $g\in \exists_{m_k^*} B$, and then
  an element $b:Q_k^*\to B$ such that $bm_k^*=g$. We then have $bj_{a_k}m_k=bm_k^*i_{a_k}=gi_{a_k}=g_{a_k}$, which means that $g_{a_k}\in \exists_{m_k} B$,
  i.e. $(B,(g_a)_{I_0})$ realises $(m_k,a_k)$.\\
  Conversely, suppose that every cocone with vertex $B$ realises an element of $X$ and let $g\in B_{P^*}$ : $g$ induces a cocone $(B,(gi_a)_{a\in I_0})$ 
  so there is $k\in\{1,\ldots,n\}$ such that $gi_{a_k}\in \exists_{m_k} B$; in other words, we have a $b\in B_{Q_k}$ such that $bm_k=gi_{a_k}$. By
  hypothesis on $m_k^*$, there exists a $b^*:Q_k^*\to B$ such that $b^*m_k^*=g$, whence $g\in \exists_{m_k^*} B$ and $B\mo[P^*,\{m_k^*\}]$.
  \end{proof}
  
  The storage of finite information in arithmetic axioms through a change of basis could lead us to the study of ``consequences'' of an arithmetic
  theory. In first order logic one finds a semantic and a syntactic notion of consequence, their equivalence being the content of Gödel's
  \emph{completeness} theorem. Here the two notions are not far from each other, and we will simply say that an axiom $\phi$ is a \emph{consequence} of an
  arithmetic theory $\t$ if for every model $B$ of $\t$, we have $B\mo\phi$. If $\C$ is any subcategory of rings, the \emph{(arithmetic) theory of $\C$}
  will be the set $Th(\C)$ of axioms satisfied in every ring $B$ of $\C$, so clearly the set of consequences $\t^\vdash$ of $\t$ is $Th(\M(\t))$, and
  clearly $\C$ is arithmetic if and only if $\C=\M(Th(\C))$. For the reader interested in Grothendieck topoi, following for example \cite{SGL} (Chapter III)
   the axioms of $\t^\vdash$ are the finite co-covers of a Grothendieck topology $\ms T$ on $\AA$, and $\M(\t)$ is equivalent to the category of points of
  the coherent topos $Sh(\AA,\ms T)$. Conversely, the collection of finite co-covers of any coherent Grothendieck topology $\ms T$ on $\AA$ is an arithmetic
  theory $\t$ such that $\t=\t^\vdash$ and the category of points of $Sh(\AA,\ms T)$ is equivalent to $\M(\t)$.

  \begin{defi}
  The \emph{negative diagram of $A$} is the set of primitive morphisms with parameters in $A$ which are \emph{not} realised in $A$.
  \end{defi}
  
  \begin{thm}\label{CARARITH}
  $\C$ is the category of models of an arithmetic theory if and only if it is closed under ultraproducts and purely
  embedded subrings. In this case, it is also closed under filtered colimits.
  \end{thm}
  \begin{proof}
  Suppose that $\C=\M(\t)$ is arithmetic.
  First, if $\ms U$ is an ultrafilter of subsets of a set $I$ and $(A_i)_I$ a family of models of $\t$, we denote by $A$
  the ultraproduct $\prod_\ms U A_i$ and we show that $A$ is in $\C$ : we have to show that $A\mo\chi$ for every $\chi\in\t$. Writing
  $\chi=[P,\{m_1:P\to Q_1,\ldots,m_n:P\to Q_n\}]$, suppose that $a\in A_P$. As $A$ is a directed colimit of the products $\prod_X A_i$ for $X\in\ms U$ and $P$ is
  finitely presented, there exists $X\in \ms U$ and a family of points $(a_i:P\to A_i)_{i\in X}$, such that $\pi_X ((a_i)_X)=a$, if we denote by $\pi_X$ the transition morphism
  of the ultraproduct. As each $A_i$ is a model of $\chi$, for each $i\in X$ there exists $k_i\in \{1,\ldots,n\}$ such that $a_i\in \exists_{m_{k_i}} A_i$.
  In particular, if we note $X_k=\{i\in X:a_i\in \exists_{m_k} A_i\}$, we have $X=\bigcup_{k=1}^n X_k$ and as $\ms U$ is a prime filter, there is
  $k_0\in\{1,\ldots,n\}$ with $X_{k_0}\in \ms U$. For each $i\in X_{k_0}$, there exists $b_i:Q_{m_{k_0}}\to A_i$ with $b_im_{k_0}=a_i$. In particular,
  we have $a=\pi_{X_{k_0}} ((a_i)_{X_{k_0}})$ and if we set $b=\pi_{X_{k_0}} ((b_i)_{X_{k_0}})$, we have $bm_{k_0}=a$, whence $a\in \exists_{m_{k_0}} A$
  and $A$ is in $\C$, which is then closed under ultraproducts. Secondly, suppose that $i:A\into B$ is a pure embedding into a ring of $\C$. With the same notations, if $a\in A_P$ consider $ia:P\to B$.
  As $B$ is a model of $\chi$, there exists $k_0\in\{1,\ldots,n\}$ such that $ia\in \exists_{m_{k_0}} B$ and as $i$ is pure we have
  $a\in \exists_{m_{k_0}} A$, so $A$ is in $\C$, which is then closed under purely embedding subrings.\\
  Reciprocally, suppose that $\C$ is closed under ultraproducts and pure subrings and let $\t=Th(\C)$ : we will show that $\C$ is the category of models of
  $\t$, so we suppose that $A\mo\t$ and apply the finiteness property with $Y$ as the negative diagram of $A$ and $X=\emptyset$. Let $Y_0=\{(m_1,a_1),\ldots,(m_n,a_n)\}$ be a finite subset
  of $Y$ and $I$ a finite subcategory of $\ms D_A$ such that $a_k\in I_0$ for $k=1,\ldots,n$. Let $(P^*,(f_a)_{I_0})$ be a colimit of
  $D_{A\lceil I}$ in $\AA^{op}$ and $f:P^*\to A$ be the universal morphism corresponding to the natural cocone $(A,(a)_{I_0})$ for $D_{A\lceil I}$. For
  every $(m_k:P_k\to Q_k,a_k)\in Y_0$, we consider the change of basis $m_k\to m_k^*$ along $f_{a_k}$ : by definition of $Y_0$ and
  Proposition~\ref{DIAMOR}, the axiom $\chi=[P^*,\{m_k^*\}]$ cannot be true in $A$, so $\chi\notin \t$.
  By definition of $\t$, there exists a ring $B$ in $\C$ which is not a model of $\chi$, and this leads to a cocone for $D_{A\lceil I}$
  with vertex $B$ avoiding $Y_0$. By the finiteness property~\ref{FPP}, there is a continuation $f:A\to B$ into an ultraproduct of rings of $\C$
   and $f$ avoids $Y$ : this exactly means that $f$ is an immersion, so by hypothesis $A$ itself is in $\C$. Finally, we have $\C=\M(\t)$.\\
  The stability of $\C=\M(\t)$ under filtered colimits is straigthforward, using an argument similar to the case of ultraproducts.
  \end{proof}
  
  The ring $A$ is trivial if and only if it satisfies the axiom $[\Z,\Z\to \mb 1]$; otherwise it satisfies the axiom $[\mb 1,\emptyset]$, so
  if $\t$ is the set of all arithmetic axioms satisfied in $A$ (the ``arithmetic theory of $A$'', $Th(A)$), the fact of $A$ being trivial or not is ``encoded'' in
  $\t$. This means that even if the two notions of existential completeness and geometric closedness do not coincide in general, if we associate to
  $A$ the category $\C=\M(Th(A))$ and if $A$ is noetherian, either $A$ is trivial (and then e.c in $\C$ if and only if g.c in $\C)$,
  or $A$ is not trivial and $\C=\C^*$ : by Corollary~\ref{CORNOETH} it satisfies the Nullstellensatz relative to $\C$ if and only if is is
  existentially closed in $\C$. This could serve as a definition for a ``geometrically closed ring``, and at this point we would be able to apply to $A$
  the methods of \emph{positive model theory} of \cite{BYP}, compatible with the present setting; we will rather wait until the next section and adopt a
  slightly different definition. We show here in the spirit of section~\ref{SEC1} how to give a set of axioms for $U_\C$ when $\C=\M(\t)$ for a given
  $\t$, as this category was proven of geometric interest in Proposition~\ref{VARIRR}.

  \begin{defi}
  A formula $F$ will be called \emph{universal} if it consists of \emph{surjective} morphisms, and an axiom will be universal if its formula is. If $\t$ is an
  arithmetic theory, the \emph{universal theory of $\t$}, noted $\tu$, it the set of universal consequences of $\t$, i.e. of all universal
  axioms satisfied in every model of $\t$.
  \end{defi}

  \begin{prop}\label{CARUN}
  If $\C=\M(\t)$, then $U_\C$ is the category of models of $\tu$.
  \end{prop}
  \begin{proof}
  Suppose that $i:A\into B$ is an embedding into a model of $\t$ and $\chi=[P,F]\in \tu$. Let
  $h_P^i:A_P\into B_P, a\mapsto ia$ be the injective map induced by $i$ and $(m:P\to Q)\in F$. If $a\in \exists_{m} A$, there is a $b:Q\to A$
  with $bm=a$, so $ibm=ia$ and $ia\in \exists_{m} B$. Conversely, if $ia\in \exists_{m} B$, let $b\in B_Q$ be such that $bm=ia$ : we
  have $bm(P)=ia(P)$, i.e. $b(Q)\subseteq i(A)$ as $m$ is surjective. If $j:i(A)\to A$ is the inverse of $i$, we have $jb:Q\to A$ and $jbm=a$, whence $a\in \exists_m A$ : we
  have proved that $\exists_m A=(h_P^i)^{-1} \exists_m B$. As $B\mo\t$, we have $B\mo\chi$, i.e. $B_P=\bigcup_{m\in F} \exists_m B$. This implies
  $A_P=(h_P^i)^{-1} B_P= \bigcup_{m\in F} (h_P^i)^{-1}\exists_m B = \bigcup_F \exists_m A$, i.e. $A\mo\chi$. Finally, we have $A\mo\tu$.\\
  Reciprocally, suppose that $A\mo\tu$ and let $Y$ be the set of all surjective primitive morphisms with parameters in $A$ \emph{avoided} by $A$. First, if
  $f:A\to B$ is an $A$-algebra avoiding $Y$ and $a\neq b\in A$, then the pair $(\Z[x,y]\to \Z[x,y]/(x-y),e_{a,b})$ (where $e_{a,b}:\Z[x,y]\to A$ is the
  evaluation of $x$ and $y$ at $a$ and $b$) is in $Y$, whence $fa\neq fb$ and $f$ is an embedding. We then only have to show that such an $A$-algebra
  exists, with $B$ a model of $\t$, i.e. to apply the Finiteness Property~\ref{FPP} to $\C$ with $X=\emptyset$.
  Let then $Y_0=\{(m_1,a_1),\ldots,(m_n,a_n)\}$ be a finite subset of $Y$ and $I$ a finite subcategory of $\ms D_A$ containing the parameters of $Y_0$. By
  way of contradiction, if every cocone for $I$ with vertex a model of $\t$ realises one of the elements of $Y_0$, then by Proposition~\ref{DIAMOR} (with
  the same notations) every model of $\t$ is a model of $\chi=[P^*,\{m_k^*\}]$. Considering the change of basis $m_k\to m_k^*$ through tensor products, it
  is easy to see that each $m_k^*$ is surjective, so $\chi\in\tu$. However, the canonical cocone $(A,I_0)$ for $I$ avoids $Y_0$ by definition of $Y$, so
  $A$ cannot be a model of $\chi$, which contradicts the hypothesis that $A\mo\tu$. We conclude that the hypothesis of the finitess property is satisfied
  and we find an $A$-algebra $f:A\to B$ which avoids $Y$, $B$ being a model of $\t$ as an ultraproduct of such by the theorem, and the
  proof is complete.
  \end{proof}

   \begin{ex}
   \begin{enumerate}
   \item The arithmetic category $\C$ of real fields is axiomatised by $\t_f\cup\t_{rd}$, where $\t_{rd}=\t_{id}\cup
   \{[\forall x_1,\ldots,x_n\ (\sum_{i=1}^n x_i^2=0)\imp (x_1=0)] : n\in \N^* \}$ is the theory of \emph{real integral domains}
   \item The ``pre-real rings'' are the models of $\t_{pr}=\{[\forall x_1,\ldots,x_n\ (1+\sum_{i=1}^n x_i^2=0)\imp \bot] : n\in \N\}$ (the case $n=0$ says that
   the ring is not trivial).
   \item The properties of real closed fields may be encoded into an arithmetic theory $\t_{rcf}$, which is $\t_{rf}$ together with
   $[\forall x\ \top\imp\exists y(y^2-x=0)\vee \exists y (y^2+x=0)]$ and the set
   $\{[\forall x_1,\ldots,x_n\ \top\imp\exists y(y^n + x_1y^{n-1}+\ldots, y_n=0)] : n\in 2\N+1\}$.
   \end{enumerate}
  \end{ex}
  
  It may be shown with similar techniques (see~\cite{HD}, Section 8.1) that arithmetic special categories are axiomatised by the set
  $\t_H$ of their universal consequences of the form $[P,F]$, where $F$ is a singleton. Such axioms we call \emph{universal Horn} and such categories
   \emph{quasivarieties of rings}, following universal algebra. In particular, starting from an arithmetic category $\C=\M(\t)$ we have an
  axiomatisation of the special category $\sc$ by $\t_H$. We leave this aside in order not to weigh down the exposition, but notice that as the real
  radical of an ideal $I$ is characterised as $\sqrt[R] I=\{a\in A:\exists m\in \N, \exists b_1,\ldots,b_n\in A, a^{2m}+b_1^2+\ldots +b_n^2 \in I\}$
  (\cite{BCR}, Proposition 4.1.7), the quasivariety $S_{\M(\t_{rf})}=\M((\t_{rcf})_H)$ from the preceding example is axiomatised by the set
  $\{[\forall x,y_1,\ldots,y_n\ (x^{2m} + \sum_{i=1}^n b_i^2 =0)\imp (x=0)]: m,n\in \N\}$. In the same spirit, it is possible to
  show that in general the set $\thu$ of consequences of $\t$ of the form $[P,\emptyset]$ defines the category $\C^-$.

  This leads us to the following characterisation of $\C^+$, reminiscent of model theory's notion of ``model companion''.
  \begin{cor}\label{CAREC}
  If $\C=\M(\t)$ is arithmetic, then for every theory $\t'$ such that $\C\subset \D=\M(\t')\subset\C^-$ we have $\C^+=\D^+$.
  \end{cor}
  \begin{proof}
  First, notice that a ring in $\C^+$ is e.c in $\C^-$, so if $A$ is a ring of $\C^+$, then any homomorphism $f:A\to B$ with $B$ in $\D$ is e.c
  and we have $\C^+\subset\D^+$. Conversely, if $A\in \D^+_0$ and $f:A\to B$ is a continuation into a ring of $\C$, by hypothesis $B$ is in $\D$
  and $f$ is an immersion, i.e $A\in \C^+$.
  \end{proof}
   \begin{rem} Mimicking with arithmetic formulas the proof of Theorem 1 of \cite{BYP}, one may show that every ring
  of an arithmetic category $\C$ has an e.c continuation, which allows to strengthen the corollary by demanding only that $\C^+\subset \D \subset \C^-$.
  \end{rem}
  In particular, if we have an axiomatisation $\t$ of a category of rings $\C$ such that $\C=\C^*$ as well as a relative Nullstellensatz for $\C$, this
  automatically tells us that g.c rings of $\C$ are the e.c models of any weaker arithmetic theory containing $\thu$ and provides a systematic way of
  proving pseudo model completeness in this context, without relying on formal model theoretical knowledge. The real Nullstellensatz (\cite{BCR}, 4.1.4) is
  an example and we recover Corollary~\ref{RCF} as a special case. This example leads us to the study of p.m.c categories which also are arithmetic, and to
  the notion of a ``geometrically closed ring''.

\section{Geometrically closed rings}\label{SEC4}

  If $p$ is a prime number, let us write $v_p$ the extension of the $p$-adic valuation on $\mb Q$ to $\mb Q_p$.
  In this last field, the relation $v_p(x)\leq v_p(y)$ is the ``$p$-adic divisibility'' $x|_p y$, which may alternatively be defined as the unique binary
  relation on $\mb Q_p$ satisfying a certain (finite) list of properties (see \cite{PR}, p.6). Following E.Robinson in~\cite{ER}, in analogy with $\mb R$
  and real closed fields this relation may be defined in $\mb Q_p$ by a primitive morphism $d_p:\Z[x,y]\to \Z[x,y,z]/(f_p(x,y,z))$,
  where $f_2(x,y,z)=z^3(x^3+2y^3)-1$ and $f_p(x,y,z)=z^2(x^2+py^2)-1$ for $p\neq 2$. Remember that a field $k$ is \emph{formally $p$-adic} if it admits a
  $p$-valuation, i.e. a valuation $v_p$ such that $v_p(p)$ is the smallest positive element of the value group and the residue field has $p$ elements. Such
  a field is \emph{$p$-adically closed} when as in $\mb Q_p$ the preceding $d_p$ induces a unique $p$-valuation $v$ such that $(k,v)$ has no proper algebraic
  $p$-valued extension (from \cite{PR}, 1).
  A (geometric) axiomatisation of $p$-adically closed fields was given in~(\cite{ER}, 1), which may be translated into an arithmetic theory $\t_{pcf}$.
  Applying Corollary 7.10 of \cite{PR} to the $\C$-radical for $\C=\M(\t_{pcf})$ and the discussion following Corollary~\ref{CAREC}, the $p$-adically closed
  fields are the e.c formally $p$-adic fields. In analogy with real closed fields, we will show that the ``open'' relation $v_p(x)<v_p(y)$ is
  definable by a \emph{finite set} of primitive morphisms, using the characterisation of pseudo model complete categories which are also arithmetic. 
  Throughout this section, $\t$ is an arithmetic theory.
  
  \begin{defi}\label{DEFZERO}
  Let $U(A)=\bigcup_{P\in \AA_0} A_P$.
  The collection of \emph{arithmetic sets of $A$} is the smallest subset $Ar(A)$ of $\ms P(U(A))$ such that for every primitive morphism $m:P\to Q$ : 
   \begin{enumerate}
    \item $A_P\in Ar(A)$ and we note $Ar_P(A)=Ar(A)\cap \ms P(A_P)$
    \item For $X\in Ar_Q(A)$, $A_m(X)\in Ar_P(A)$
    \item For $X,Y\in Ar_P(A)$, $X\cap Y, X\cup Y\in Ar_P(A)$
   \end{enumerate}
  We will say that $\t$ is \emph{(positively) model complete} if $\M(\t)$ is pseudo model complete.
  \end{defi}
  If $F$ is an arithmetic formula on the f.p ring $Q$, it ``defines'' a subset of $A_Q$, namely $A_F=\bigcup_{n\in F}\exists_n A$. Consider a primitive
  morphism $m:P\to Q$ : we have $A_m(A_F)=A_m (\bigcup_{n\in F}\exists_n A )= \bigcup_{n\in F} \exists_{nm} A$, hence $A_m(A_F)$ is ``defined'' by the
  formula $m^*F=\{nm:n\in F\}$. If $G$ is another formula on $Q$, it is not difficult to see that $A_F\cap A_G$ is defined in $A_Q$ by the formula
  $F\otimes G=\{m\otimes n : Q\to R\otimes S | (m,n)\in F\times G\}$ (where $m\otimes n :Q\to R\otimes S$ are induced by the tensor product). This means that the arithmetic
  subsets of $A_Q$ are the subsets ``defined'' by an arithmetic formula.
  
  \begin{thm}\label{CARPMC}
  Let $\C=\M(\t)$. The following properties are equivalent :
  \begin{enumerate}
   \item $\t$ is model complete
   \item For every primitive morphism $m:P\to Q$, there exists a finite subset $X$ of $Res_\C [m]$ such that for every model $B$ of $\t$, $B_P-\exists_m B
   = \bigcup_{n\in X} \exists_n B$.
  \end{enumerate}
  In other words, an arithmetic category is pseudo model complete if and only if for every $P$, $Ar_P(A)$ is the Boolean algebra of clopen sets of a compact
  topology.
  \end{thm}
  \begin{proof}
  $(1)\imp(2)$ Consider the set of morphisms with parameters $Y=\{(m,1_P)\}\cup\{(n:P\to R_n,1_P) : n\in Res_\C [m]\}$.
  If $f:P\to B$ is a continuation with $B\mo\t$ avoiding the whole $Y$, by hypothesis $B$ is an e.c ring of $\C$ in which
  $\bigcup_{n\in Res_\C [m]}\exists_n B \neq B_Q - \exists_m B$, contradicting Lemma~\ref{RESCOMP}. This means that such a continuation does not exist,
  and as $\C$ is closed under ultraproducts by Theorem~\ref{CARARITH}, by the finiteness property~\ref{FPP} there is a finite subset $Y_0$ of $Y$ and a
  finite subcategory $I$ of $\ms D_P$ such that $1_P\in I_0$ and no cocone for $\ms D_{P\lceil I}$ with vertex in $\C$ avoids $Y_0$. We may suppose that
  $I$ is generated by $P$ and $(m,1_P)\in Y_0$, so by Proposition~\ref{DIAMOR} we have $\t\mo [P,Y_0]$, whence in every model $A$ of $\t$ we have
  $A_P-\exists_m A =\bigcup_{n\in Y_0-\{m\}} \exists_n A$ and $X=Y_0-\{m\}$ is the desired finite subset of $Res_\C[m]$.\\
  $(2)\imp(1)$ Let $f:B\to C$ be an homomorphism in $\C$, $m:P\to Q$ primitive and $b\in B_P$ such that $fb\in \exists_m C$. By hypothesis let $X$ be a
  finite subset of $Res_\C [m]$ defining the complement of $m$ modulo $\C$. For every $n\in X$ we have $fb\notin\exists_n C$, and as $f$ is an homomorphism
  we get $b\notin\exists_n B$, whence $b\in\exists_m B$ because $B\in \C$. This shows that $f$ is existentially closed, which means that $\C$ is pseudo
  model complete, i.e $\t$ is model complete.
  \end{proof}

  If $\C$ is arithmetic, so is $\C^*$, and if $\C$ is p.m.c, so is $\C^*$. By the preceding characterisation and Corollary~\ref{CARGEC},
  if $\t$ is model complete and $A$ is in $\C=\M(\t)$, either $A$ is trivial (in which case it is geometrically closed in $\C$), or $A$
  is geometrically closed in $\C^*$. If now $\t=Th(A)$, as precedently discussed the information on the triviality of $A$ is encoded in $\t$,
  so $A$ is existentially closed in $\C$ if and only if it is geometrically closed in $\C$ in every case. 

  \begin{defi}
  A ring $A$ is \emph{geometrically closed} if its arithmetic theory $\t=Th(A))$ is model complete, equivalently
   if $A$ is a model of a certain model complete arithmetic theory.
  \end{defi}
  
  The second clause of the definition takes care of a.c fields : indeed, the model completeness of $\t_{acf}$ follows at once from Theorem~\ref{CAC},
  but the theory of a given a.c field encodes its characteristic, while $\t_{acf}$ does not. Every real closed field and every $p$-adically closed
  field is geometrically closed as well. We recall that a subset $X$ of an $A_P$ is
  \emph{constructible} if it is a Boolean combination of affine varieties. Chevalley's theorem states that in an
  a.c field the image of a constructible set under a regular morphism is again constructible. In such a field, inequations themselves
  are the projections of affine varieties (for $P(X)\neq 0$, take the projection along the $y$-axis of $\ms Z(P(X)y-1)$), which means that
  every subset $X$ of $U(A)$ obtained by boolean combinations and projections of affine varieties, may be obtained by
  projections (and unions) of affine varieties. We will establish this property in any geometrically closed ring.

  \begin{defi}
  The collection of \emph{definable sets with parameters in $A$}  is the smallest subset $Def(A)$ of $\ms P(U(A))$ such that for every primitive
  morphism $m:P\to Q$, $Def_P(A)=Def(A)\cap \ms P(A_P)$ is a Boolean subalgebra of $\ms P(A_P)$ and for every $X\in Def_Q(A)$, $A_m(X)\in Def_P(A)$.\\
  If, as in Definition~\ref{DEFZERO}, we omit the complements, we get the collection $Def^+(A)$ of \emph{positively definable sets} of $A$.
  \end{defi}

  \begin{lem}
  If $m:P\to R$, $n:Q\to R$ are primitive, $X\in Ar_P(A)$ and $Y\in Ar_Q(A)$, then the fibered product
  $X*_RY=\{a\otimes b \in A_{P\otimes_R Q} : a\in X \ \& \ b\in Y\}$ is arithmetic. In particular,   the set $A_m^{-1}(X)$ is arithmetic.
  \end{lem}
  \begin{proof}
  We first suppose that $Y=A_Q$. If $X=A_P$, $X*_R A_Q$ is the whole $A_{P\otimes_ R Q}$, which is arithmetic, while the clauses for
  $\cap$ and $\cup$ obviously preserve the property. If $X=A_l(Z)$, where $l:P\to S$ and $Z\in Ar_S(A)$, suppose that $Z*_R A_Q$ is arithmetic.
  We have $X*_R A_Q=A_l(Z)*_R A_Q=A_u (Z*_R A_Q)$, where $u:S\otimes_R Q\to P\otimes_R Q$ is universal, and this takes care
  of projections : $X*_R A_Q$ is arithmetic for every $X\in Ar_P(A)$. In general, we have $X*_R Y= X*_R A_Q \cap A_P*_R Y$, and this last is arithmetic as
  an intersection, if $X$ and $Y$ are.
  \end{proof}

  The functor $hom(-,A)$, being left exact on $\AA$, turns tensor products into fibered products : the preceding lemma extends this to arithmetic sets, and
  if we forget $R$ we get an inside product $X*Y=X*_\Z Y$. If $U\in Ar_P(A)$ and $V\in Ar_Q(A)$, a map $\phi:U\to V$ will be \emph{arithmetic} if its
  virtual
  graph $Gr^*(\phi)=\{a\otimes b\in U*V : b=\phi(a)\}$ is. For such a map, if $b\in V$ we have
  $\phi^{-1}(b)=A_p(Gr^*(\phi)\cap A_q^{-1}(b))$, where $p,q:P,Q\to P\otimes Q$. As $A_q^{-1}(b)$ is an affine variety (see the proof of Lemma~\ref{CARIM}) and every arithmetic set is
  obviously positively definable, any set of the form $A_p(Y\cap A_q^{-1}(b))$, where $Y\subseteq A_{P\otimes Q}$ is arithmetic, is positively definable
  and the fiber of an arithmetic map. In model theory definable sets with parameters are obtained by adding parameters to basic definable sets, which in turn are defined by
  first order formulas. Here we started from primitive ring homomorphisms, which does not allow us to speak of boolean combination of arithmetic formulas,
  but we have to check that our definable sets with paramaters are also obtained by adding parameters in a certain sense.

  \begin{lem}
  A subset $X$ of $A_P$ is positively definable if and only if there exists an f.p ring $Q$, an arithmetic subset $Y$ of $A_{P\otimes Q}$ and a
  parameter $b\in A_Q$ such that $X=\{a\in A_P : a\otimes b \in Y\}$.
  \end{lem}
  \begin{proof}
  For $P$, $Q$, $Y$ and $b$ as in the statement, the set $\{a\in A_P : a\otimes b \in Y\}=A_p(Y\cap A_q^{-1}(b))$ is in $Def_P^+(A)$ by
  the preceding discussion, so we only need to show that positively definable sets have this form. If $V$ is an affine variety of $A_P$, there is a
  primitive $m:Q\to P$ and a parameter $b\in A_Q$ such that $V=A_m^{-1}(b)$. Consider the canonical morphisms $p:P\to P\otimes Q$ and $q:Q\to P\otimes Q$
  : we have two parallel arrows $q,pm:Q\to P\otimes Q$, of which we may find a primitive coequaliser $e:P\otimes Q\to E$, and it is easy to check that
  $V=\{a\in A_P : a\otimes b\in \exists_e A\}$; as $\exists_e A$ is arithmetic, every affine variety has the desired form. Suppose that $X_i=\{a\in A_P:
  a\otimes b_i\in Y_i\}$ for $Y_i\in Ar_{P\otimes Q_i} (A)$ and $b_i\in A_{Q_i}$, $i=1,2$. Write $b_1\otimes b_2\in A_{Q_1\otimes Q_2}$, $p:P\to P\otimes
  (Q_1\otimes Q_2)$, and $k_i:Q_i\to P\otimes(Q_1\otimes Q_2)$, $q_i:Q_i\to P\otimes Q_i$ the canonical morphisms. We have
  $X_1\cap X_2=\{a\in A_P: a\otimes(b_1\otimes b_2) \in A_{k_1}^{-1} A_{q_1}(Y_1) \cap A_{k_2}^{-1} A_{q_2}(Y_2)\}$ and similarly for $X_1\cup X_2$,
  whence these two subsets of $A_P$ have the desired form by the preceding lemma. Finally, if $m:P\to Q$ is primitive, $X=A_m(Y)$ and
  $Y=\{b\in A_Q: b\otimes c\in Z\}$, for $Z\in Ar_{Q\otimes R}(A)$ and $c\in A_R$, consider the canonical morphisms $q:Q\to Q\otimes R$, $r:R\to Q\otimes R$, $p:P\to P\otimes R$ and $\rho:R\to P\otimes R$. There exists a unique $u:P\otimes R
  \to Q\otimes R$ such that $up=qm$ and $u\rho=r$, and it may be checked that $X=\{a\in A_P : a\otimes c\in A_u(Z)\}$, whence the property is closed
  under projections, because $A_u(Z)$ is arithmetic. 
  \end{proof}

  \begin{cor}
  If $A$ is geometrically closed, then every definable set of $A$ is positively definable.
  \end{cor}
  \begin{proof}
  If $P\in \AA_0$ and $X\in Def_P^+(A)$ is a positively definable variety, let $Q$, $Y$ and $b$ be as in the preceding lemma, so that $X=\{a\in A_P : a\otimes b\in Y\}$.
  By hypothesis $\t=Th(A)$ is model complete, so by Theorem~\ref{CARPMC} $Y$ has an arithmetic complement $Z=A_{P\otimes Q}-Y$. This means, again by the lemma, that
  $A_P- X= \{a\in A_P :a\otimes b\notin Y\}=\{a\in A_P:a\otimes b \in Z\}$ is positively definable, whence we may omit the complements in the definition
  of $Def(A)$, and $Def(A)=Def^+(A)$.
  \end{proof}

  In a real closed field $k$, the definable sets are called ``semi-algebraic'' (\cite{BCR}, 2) and a great part of real algebraic geometry deals with
  the study of those sets. Every semi-algebraic set is obtained by ``adding parameters to an arithmetic'' set, which was patent when we discussed the
  ``definability'' of the real order with a primitive morphism. The Euclidean topology on $k$ is defined by this order as in $\mb R$, and the
  Tarski-Seidenberg principe is an analogue of Chevalley's theorem which states that every semi-algebraic set in $k$ is a boolean combination of
  \emph{open} (or equivalently closed) definable sets. Similarly, in $p$-adically closed fields definable sets may be decomposed into boolean combinations
  of certain definable sets, among which we find the closed sets for the topology defined bu the $p$-adic valuation (see \cite{BE2}). We conjecture that
  those are exactly the closed sets for the $p$-adic topology. As these relations are defined in these geometrically closed rings by primitive morphisms,
  the general problem of structure for the definable sets of a given (noetherian, integral) geometrically closed ring $A$ would seem to be the
  identification of a topology $\ms T$ finer than Zariski, generated by an arithmetically definable basis and such that every definable set is a boolean
  combination of open (or closed) definable sets of $\ms T$.\\
  The concept of a g.c ring is indeed not restricted to the (algebraic analogues of the) fields $\mb C$ and $\wt{\mb F}_p$'s, $\mb R$ and $\mb Q_p$'s. For instance, the work of McKenna in~\cite{MK}
  mentions continuously many examples of model complete theories of fields in positive or zero characteristic, and simple model theoretic arguments show
  that every model complete theory of fields in the logical sense defines a model complete arithmetic category. The concepts expounded here are compatible with
  McKenna's results outside the scope of polynomial rings. 
  As definable sets are of intrisic interest in g.c rings, the other scope of application is model theory in an algebraic fashion. More
  fundamentally, as every finite quotient of $\Z$ is geometrically closed, it seems that a most natural question one should ask about an infinite geometrically
  closed ring is whether it is a field, as these are the only examples we have exhibited.

\end{document}